\documentclass[11pt]{amsart}

\usepackage{enumerate}
\usepackage{amssymb}
\usepackage{amsmath,amscd}
\usepackage[colorinlistoftodos,disable]{todonotes}
\usepackage{amsthm}
\usepackage{hyperref}
 \usepackage[capitalise,noabbrev]{cleveref}
\newcommand{\F}{\mathcal{F}}
\newcommand{\RR}{\mathbb{R}}
\newcommand{\CC}{\mathbb{C}}
\newcommand{\HH}{\mathbb{H}}
\newcommand{\Z}{\mathbb{Z}}
\newcommand{\B}{\mathcal{B}}
\renewcommand{\L}{\mathcal{L}}
\newcommand{\reg}{\operatorname{reg}}
\newcommand{\Av}{\operatorname{Av}}
\newcommand{\vol}{\operatorname{vol}}

\newtheorem{theorem}{Theorem}
\newtheorem{corollary}[theorem]{Corollary}
\newtheorem{lemma}[theorem]{Lemma}
\newtheorem{proposition}[theorem]{Proposition}

\newtheorem{maintheorem}{Theorem}

\newtheorem{maincorollary}[maintheorem]{Corollary}

\theoremstyle{definition}
\newtheorem{definition}[theorem]{Definition}

\theoremstyle{remark}
\newtheorem{remark}[theorem]{Remark}
\newtheorem{example}[theorem]{Example}
\newtheorem{question}[theorem]{Question}

\title{Manifold submetries from compact homogeneous spaces}
\date{}

\author[S. Lin]{Samuel Lin}
\address{Colby College\newline
\indent Department of Mathematics\newline
\indent 4000 Mayflower Hill\newline
\indent Waterville, ME 04901, USA}
\email{samlin@colby.edu}

\author[R. A. E. Mendes]{Ricardo A. E. Mendes}
\address{University of Oklahoma\newline
\indent Department of Mathematics\newline
\indent 601 Elm Ave\newline
\indent Norman, OK, 73019-3103, USA}
\email{ricardo.mendes@ou.edu}

\author[M. Radeschi]{Marco Radeschi}
\address{Universit\`a degli Studi di Torino\newline
\indent Departimento di Matematica ``G. Peano''\newline
\indent Via Carlo Alberto, 10\newline
\indent 10123 Torino (TO), Italy}
\email{marco.radeschi@unito.it}

\thanks{R.~M. has been supported by DMS-2506409, DMS-2005373, and the Dodge Family College of Arts and Sciences Junior Faculty Summer Fellowship of the University of Oklahoma. Part of this work was done during a visit of R.~M. to M.~R. at the Universit\`a degli Studi di Torino, and R.~M. is grateful for the hospitality.}

\begin{document}

\begin{abstract}
We show that singular Riemannian foliations, or, more generally, manifold submetries, defined on a compact normal homogeneous space, have algebraic nature. 
%By definition, this means there exists a finite set $\{\rho_i\}$ of eigenfunctions of the Laplace--Beltrami operator such that the leaves of the foliation (or, more generally, the fibers of the manifold submetry) are precisely the common level sets of the $\{\rho_i\}$. 
Moreover, in this case there exists a one-to-one correspondence between algebras of algebraic functions preserved by the Laplace--Beltrami operator, and manifold submetries. 

A key intermediate result is that, for any manifold submetry on a compact normal homogeneous space, the vector field given by the mean curvature of the fibers is basic, in the sense that it is related to a vector field in the base.
\end{abstract}

\maketitle

\section{Introduction}
A \emph{manifold submetry} is defined as a map $\sigma\colon M\to X$, where $(M,g)$ is a Riemannian manifold and $X$ a metric space, which, in addition to being a \emph{submetry}  (that is, $\sigma$ takes metric balls to metric balls of the same radius), has fibers that are smooth, possibly disconnected, submanifolds. Submetries and manifold submetries were originally defined in \cite{Berestovskii87} and \cite{CG16}, respectively. See \cref{prelimsub} for more details. 

Manifold submetries generalize many classical objects: Riemannian submersions, (the quotient maps of) proper isometric group actions and singular Riemannian foliations with closed leaves, including  isoparametric foliations. Moreover, submetries and manifold submetries also appear naturally in rigidity questions and conjecturally in collapsing with lower curvature bounds. We point the interested reader to the  list of references in \cite[Introduction]{KL22}.

An equivalent, and perhaps more natural, object is a partition $\F$ of $M$ into equidistant subsets: The fibers of a submetry $\sigma\colon M\to X$ are equidistant, and, conversely, given a partition $\F$ into equidistant closed subsets (``leaves''), there is a unique metric on the ``leaf space'' $M/\F$ making the natural projection map $\sigma\colon M\to M/\F$ into a submetry. This leads to a natural equivalence relation between (manifold) submetries, namely, $\sigma\colon M\to X$ and $\sigma'\colon M\to X'$ are equivalent if their fibers define the same partition.

In the case where $M$ is a round sphere, every manifold submetry is an \emph{algebraic} object. This has been established in \cite{LR18, MR20}, and results in close ties (including applications) to classical Invariant Theory, see \cite{MR20, MR23, LMR23}. The main goal of the present article is to extend these results to the case where $M$ is a compact homogeneous space.

In this article, a connected, compact Riemannian manifold will be called \emph{homogeneous} when its group of isometries acts transitively, and \emph{normal homogeneous} when, for some compact Lie group $G$ acting transitively on $M$, some bi-invariant metric on $G$ induces the given metric on $M$. The latter include all compact symmetric spaces. Key to the generalization of the algebraicity of manifold submetries from spheres to compact homogeneous space is the selection of an appropriate algebra of functions generalizing polynomials. Given a compact homogeneous space $M$, we define the algebra of ``algebraic functions'' $\RR[M]\subset C^\infty(M)$ to be the span of the eigenfunctions of the Laplace--Beltrami operator. See \cref{prelimhom} for more details, including a justification of the word ``algebraic''.

Our first main result generalizes the main results in \cite{LR18, MR20, MR23} to compact normal homogeneous spaces:
\begin{maintheorem}\label{normal}
Let $(M,g)$ be a compact normal homogeneous space.
\begin{enumerate}[(a)]
\item Every manifold submetry $\sigma\colon M\to X$ is \emph{algebraic}, in the sense that there exists a finite set $\rho_1, \ldots, \rho_k\in\RR[M]$ of basic (that is, constant on the fibers of $\sigma$) eigenfunctions such that the map $\rho=(\rho_1, \ldots, \rho_k)\colon M\to\RR^k$ descends to a homeomorphism from $X$ to the semi-algebraic set $\rho(M)\subset \RR^k$. In particular, each fiber of $\sigma$ is an algebraic variety, namely a level set of $\rho$.
\item There exists a one-to-one correspondence between manifold submetries defined on $M$ (modulo the natural equivalence relation) and sub-algebras of $\RR[M]$ which are preserved by the Laplace--Beltrami operator.
\end{enumerate}
\end{maintheorem}

\cref{normal} is, in fact, a corollary of more general results (\cref{thmalgebraic} for (a) and \cref{1-1correspondence} for  (b)), where $M$ is a compact (not necessarily normal) homogeneous space and manifold submetries  $\sigma\colon M\to X$ are required to have \emph{basic mean curvature}. Roughly speaking, this means that the mean curvature vectors of the fibers of $\sigma$ form a vector field which is $\sigma$-related to a vector field on $X$. \cref{normal} then follows from \cref{thmalgebraic}, \cref{1-1correspondence}, and the following:
\begin{maintheorem}\label{basicmean}
Let $(M,g)$ be a compact normal homogeneous space, and let $\sigma\colon M\to X$ be a manifold submetry. Then $\sigma$ has basic mean curvature.
\end{maintheorem}

Prior to \cref{basicmean}, the basic-mean-curvature property was only known when the  singular Riemannian foliation is given by the orbits of an isometric action, or when the ambient space is one of: spheres, Euclidean spaces \cite{AR15}, or $\CC P^n, \HH P^n$ \cite{AR16}. \cref{basicmean} significantly expands this list of spaces to all compact normal homogeneous spaces (in particular,  all compact symmetric spaces).

The proof of \cref{basicmean} follows from \cref{shape}, which is an explicit formula for the mean curvature of an arbitrary submanifold of a compact normal homogeneous space in terms of its focal distances, and which we believe is of independent interest. In turn, the proof of \cref{shape} uses the existence of a Riemannian submersion from a Hilbert space onto the compact normal homogeneous space, and the theory developed around this submersion by several authors, including in \cite{TT95, HLO06}. To the best of our knowledge, until now this theory has only been used in the context of isoparametric foliations, so the application to \cref{basicmean} is novel.

Due to the long-standing interest in the classification of Riemannian submersions from Lie groups with bi-invariant metric (see \cite[Problem 5.4]{Grove02}, and \cref{Q:regular} below), we point out the following immediate consequence of \cref{normal}:
\begin{maincorollary}
Riemannian submersions from compact Lie groups with bi-invariant metric are algebraic.
\end{maincorollary}

\subsection*{Open questions}\ 

The 1-1 correspondence between Laplacian algebras and manifold submetries in \cref{normal}(ii) leads to a dictionary between algebraic properties of Laplacian algebras, and  geometric properties of the corresponding manifold submetries.  See \cite{MR20, MR23, LMR23} for a few entries in this dictionary in the spherical case. We propose:
\begin{question}\label{Q:regular}
Let $(M,g)$ be a compact homogeneous space. Characterize (algebraically) the Laplacian sub-algebras of $\RR[M]$ which correspond to \emph{regular} Riemannian foliations.
\end{question}
A potential application is the study of Riemannian submersions from compact Lie groups with bi-invariant metric, a question proposed in \cite[Problem 5.4]{Grove02}. In the case where the fibers are totally geodesic, it is expected that only coset fibrations occur, and a proof was proposed in \cite{Speranca17}, but in the general case a wealth of non-coset examples were found in \cite{KS12}. In the case $M=S^n$, the only Riemannian submersions are the Hopf fibrations, see \cite{GG88, Wilking01, LW16} (with older partial results by Escobales and Ranjan). Even in this case, we believe it would be interesting to produce an alternative proof using an answer to \cref{Q:regular}.

\begin{question}\label{Q:fg}
Let $(M,g)$ be a compact homogeneous space, and $A\subset \RR[M]$  a Laplacian subalgebra. Is $A$ finitely generated?
\end{question} 
In the spherical case, the answer is yes, and the proof in \cite{MR20} uses an extension of the foliation on $S^n$ to $\RR^{n+1}$, where the associated Laplacian algebra is graded. A strengthening of (and possible approach to) \cref{Q:fg} is to investigate whether, given a transitive isometric action of a Lie group $G$ on $M$, and an isometric $G$-equivariant embedding of $M$ into an orthogonal $G$-representation $V$, there exists an extension of the given manifold submetry to an infinitesimal manifold submetry on $V$.

\begin{question}\label{Q:product}
For what closed Riemannian manifolds $M$  is $\oplus_\lambda E_\lambda$ (the span of all eigenfunctions of the Laplace--Beltrami operator) closed under multiplication? That is, when is the product of any two eigenfunctions a \emph{finite} sum of eigenfunctions?
\end{question}
\cref{R[M]} implies that the answer is yes for $M$ homogeneous. If $M$ satisfies this property, it is not hard to see that so does $M/\Gamma$, where $\Gamma$ is a finite group acting freely by isometries. Thus some inhomogeneous spaces, such as the Klein bottle, also have this property. It seems likely that only very special compact Riemannian manifolds satisfy this property, but the authors are not aware of many cases where this property is known to fail. One such case is that of compact hyperbolic surfaces, see 
\cite{MOproduct}.

\begin{question}
For what Riemannian manifolds $M$ is it true that, for every submanifold $L\subset M$, and point $p\in L$, the mean curvature vector $H(p)$ at $p$ is determined by the focal distances in the normal directions to $L$ at $p$?
\end{question}
These include all compact normal homogeneous spaces by \cref{shape}. One can rewrite this question in an abstract way along a single geodesic, in terms only of the curvature endomorphism. There may be interesting relations with complex analysis, via partial fraction decomposition, in analogy with \cref{E:Euler}.

\subsection*{Overview of the proofs}\

{\bf \cref{basicmean}: } Let $M$ be a complete Riemannian manifold, and $\sigma\colon M\to X$ a manifold submetry. By arguments using Wilking's Transverse Jacobi equation, the ``focal data'' of $\sigma$ are basic (see \cref{basicfocaldata}). This means that, along two geodesics starting normally to the same regular $\sigma$-fiber $L$, and with initial velocities that map to the same vector by the differential $d\sigma$, the collection of $L$-focal distances, counted with multiplicities, are the same. 

Now assume $M$ is a compact normal homogeneous space. Then, by \cref{shape}, the collection of focal distances along normal geodesics to any submanifold $L$ of $M$ determine its mean curvature. In particular, $\sigma$ has basic mean curvature.

The proof of \cref{shape} relies on the existence, and properties, of a Riemannian submersion with minimal fibers (in a ``regularized'' sense) from a separable Hilbert space $V$ onto $M$, which we extract from \cite{HLO06}. Roughly speaking, the focal distances along a normal geodesic $c$ to $L$ do not change when one passes to the inverse image $\hat{L}\subset V$, and a horizontal lift $\hat{c}$ of the geodesic. Since the fibers of $V\to M$ are minimal, the trace of the corresponding shape operators also do not change when one passes from $L$ to $\hat{L}$. Finally, since $V$ is a vector space, the eigenvalues of the shape operator of $\hat{L}$ are simply the reciprocals of the focal distances.

{\bf \cref{normal}(a):} Let $M$ be a compact normal homogenous space, and $\sigma\colon M\to X$ a manifold submetry. Define the averaging operator $\Av\colon L^2(M)\to L^2(M)$ by defining $\Av(f)(p)$ to be the average value of $f$ on the unique $\sigma$-fiber $L_p$ going through $p$. Since, by \cref{basicmean}, $\sigma$ has basic mean curvature, arguments as in \cite{LR18} involving elliptic regularity imply that $\Av$ takes smooth functions to smooth functions, and commutes with the Laplace--Beltrami operator $\Delta$ (\cref{average}). Since the span of the eigenfunctions of $\Delta$ (which we denote $\RR[M]$) is dense in $C^0(M)$ (\cref{supnorm}), the averaging operator can be used to show that the algebra $\B(\sigma)$ spanned by all $\sigma$-basic eigenfunctions separates $\sigma$-fibers. Using the fact that $\RR[M]$ is Noetherian,  we extract a finite subset $\{\rho_1, \ldots, \rho_k\}\subset\B(\sigma)$ of basic eigenfunctions which still separate fibers (\cref{lemmasep}), and it is routine to check (\cref{P:algebraic}) that these eigenfunctions satisfy the properties listed in \cref{normal}(a).

{\bf \cref{normal}(b):} The key step is to construct, given an arbitrary sub-algebra $A\subset \RR[M]$ preserved by $\Delta$, a manifold submetry on $M$ whose algebra of basic algebraic functions coincides with $A$. This is done in steps similar to \cite{MR20}, where the case $M$ a round sphere is treated. It starts with  purely algebraic facts about $A$, in particular that $A$ is maximal (in the sense that any strictly larger sub-algebra separates more points), that $A$ admits a Reynolds operator (an algebraic analogue of the averaging operator), and that $A=F(A)\cap \RR[M]$, where $F(A)$ denotes the field of fractions of $A$.

The geometric construction starts from a finite separating set $\{\rho_1, \ldots, \rho_k\}\subset A$, from which a Riemannian submersion with basic mean curvature $\rho_0\colon M_0\to X_0\subset\RR^k$ is defined, where  $M_0\subset M$ is open and full-measure, and $X_0$ is a smooth manifold with an appropriate Riemannian metric. Taking closure, we obtain a submetry $\hat{\rho}\colon M\to\hat{X}$, which we prove is a manifold submetry using \cite{LW24}. Using that $A$ is maximal and the fibers of $\rho_0$ and $\hat{\rho}$ agree on $M_0$, we prove that the sub-algebras $A$ and $\B(\hat{\rho})$ have the same field of fractions, and this implies  $A=\B(\hat{\rho})$.

\subsection*{Acknowledgements} We thank Alexander Lytchak for  the idea of using the infinite-dimensional theory of submanifolds to prove \cref{basicmean}.

\section{Preliminaries}
\subsection{Spectrum of the Laplacian}
Let $(M,g)$ be a closed Riemannian manifold. The Laplace--Beltrami operator (or ``Laplacian'') on functions  is defined by
\[\Delta (f) = -\operatorname{div}(\nabla f) \]
for $f\in C^\infty(M)$.

We denote the eigenvalues of $\Delta$ by 
\[0=\lambda_0<\lambda_1<\lambda_2<\cdots\]
where $\lambda_i\to\infty$ as $i\to\infty$. For an eigenvalue $\lambda$, we denote by $E_\lambda$ the corresponding eigenspace, and note that it is finite-dimensional. If we select an $L^2$-orthonormal basis of each $E_{\lambda}$, their union over all $\lambda$ forms an $L^2$-orthonormal basis of $L^2(M)$.

Any isometry of $M$ induces a linear map on $C^\infty(M)$ which preserves each eigenspace. Thus, given an isometric group action of $G$ on $M$, each eigenspace has the natural structure of a $G$-representation. 

We will need the following result, which seems to be well-known to experts. Since the authors have not been able to locate a reference in the literature, a proof is provided.
\begin{lemma}\label{supnorm}
Let $(M,g)$ be a closed Riemannian manifold. Then $\oplus_\lambda E_\lambda$ (the span of all eigenfunctions of $\Delta$) is dense in $C^0(M)$ in the $\sup$ norm.
\end{lemma}
\begin{proof}
By the Stone--Weierstrass theorem (see, e.g. \cite[Theorem 7.32 on page 162]{Rudin}), $C^\infty(M)$ is dense in $C^0(M)$ in the sup norm, so it suffices to show $\oplus_\lambda E_\lambda$ is dense in $C^\infty(M)$.

Let $f\in C^\infty(M)$, and write $f=\sum_{k=0}^\infty a_k \phi_k$, where $\{\phi_k\}$ is an $L^2$-orthonormal basis of eigenfunctions, with eigenvalues $\lambda_{i(k)}$. Using self-adjointness of $\Delta$, we have $a_k=\langle f,  \phi_k\rangle = \langle \Delta^j f,  \lambda_{i(k)}^{-j}\phi_k\rangle $, and, in particular, $|a_k|\leq \|\Delta^j f \|_{L^2} \lambda^{-j}_{i(k)}$, for every choice of $j$. 

It follows from Weyl's Law (see \cite[Section I.3]{Chavel}) that $\lambda_{i(k)}\sim C k^{2/n}$ (where $n=\dim M$). This implies  that $|a_k|=O(k^{-N})$ for every choice of $N>0$, because we can take $j> Nn/2$ in the inequality above.

Since the sup norm satisfies $\|\phi_k\|_{L^\infty}\leq C' \lambda_{i(k)}^{(n-1)/4}$ (see e.g., \cite[Abstract]{SZ02}), we can take $N$ large enough so that $\frac{2}{n}\frac{n-1}{4}-N<-1$, and use Weyl's Law again to conclude that the series $\sum_{k=0}^\infty |a_k| \| \phi_k\|_{L^\infty}$ converges.

Therefore the partial sums of $\sum_{k=0}^\infty a_k \phi_k$ form a sequence in $\oplus_\lambda E_\lambda$ which converges  in sup norm to $f$.
\end{proof}

\subsection{Algebraic functions on compact homogeneous spaces}\label{prelimhom}
In this section we fix notations and recall basic facts about homogeneous spaces.

\begin{definition}
We say a connected closed Riemannian manifold $(M,g)$ is a \emph{compact homogeneous space} when the natural action by the isometry group is transitive. We say a compact homogeneous space $(M,g)$ is a \emph{compact normal homogeneous space} when there is a (compact) Lie group $G$ acting transitively by isometries on $(M,g)$ and a bi-invariant metric $Q$ on $G$ such that the natural quotient map $(G,Q)\to (M,g)$ is a Riemannian submersion.
\end{definition}

\begin{definition}
Let $(M,g)$ be a compact homogeneous space, and $f\colon M\to\RR$ a function. We say $f$ is \emph{algebraic} when $f$ is a finite sum of eigenfunctions of the Laplace--Beltrami operator $\Delta$. We denote by $\RR[M]$ the ($\RR$-vector space\footnote{Actually an $\RR$-algebra, see following couple of lemmas.}) of all algebraic functions. In other words, $\RR[M]=\bigoplus_\lambda E_\lambda$, where $E_\lambda$ denotes the eigenspace associated to the eigenvalue $\lambda$.
\end{definition}

The following facts justify the name ``algebraic'' in the previous definition:
\begin{lemma}\label{algebraicfunctions}
Let $(M,g)$ be a compact homogeneous space, and $G$ a compact Lie group acting transitively by isometries on $M$. Then there exists a $G$-equivariant embedding $\varphi\colon M\to\RR^N$, where $G$ acts linearly on $\RR^N$. Moreover, for any such choice of embedding, and any continuous function $f\colon M\to \RR$, the following are equivalent:
\begin{enumerate}[(a)]
\item $f$ is algebraic.
\item $f$ is $G$-finite, that is, the $G$-orbit of $f$ spans a finite-dimensional vector space.
\item $f=P\circ \varphi$ for some polynomial map $P\colon\RR^N\to \RR$.
\end{enumerate}
\end{lemma}
\begin{proof}
By \cite{Mostow57, Palais57}, a $G$-equivariant embedding $\varphi$ always exists. Fix one such embedding.

The equivalence between (b) and (c) follows from \cite[Lemma 3.1]{Helgason63}. 

(a)$\implies$(b): Assume $f$ is algebraic. Since each eigenspace of $\Delta$ is a finite-dimensional $G$-invariant subspace, it follows that $f$ is $G$-finite.

(b)$\implies$(a):  let $\F(M)$ denote the algebra of $G$-finite functions, fix $p\in M$, and consider $\pi\colon G\to M$ given by $p\mapsto g\cdot p$.  Since $\pi$ is $G$-equivariant and surjective, the pull-back map $\pi^*$ is a $G$-equivariant linear embedding of $\F(M)$ into the algebra $\F(G)$ of $G$-finite (that is, ``representative'') functions on $G$. By \cite[III Prop. 1.5]{BtD}, $\F(G)$ is the direct sum of its isotypical components, each of which is finite-dimensional. Therefore  $\F(M)$ is the direct sum of its isotypical components, each of which is finite-dimensional.

Since $\Delta$ is $G$-equivariant, it preserves $\F(M)$, and it preserves each of its isotypical components. Let $f\in\F(M)$ be a $G$-finite function. Then $f=f_1+\cdots + f_k$, where each $f_i$ lies in an isotypical component. By the (finite-dimensional) spectral theorem, each $f_i$ is a finite sum of eigenfunctions of $\Delta$, and therefore $f$ is algebraic.
\end{proof}

\begin{lemma}\label{R[M]}
Let $(M,g)$ be a compact homogeneous space. Then $\RR[M]$ is a finitely-generated $\RR$-algebra, which is dense in $C^0(M)$ in the supremum norm.
\end{lemma}
\begin{proof}
Let $G$ be the isometry group of $M$, and choose any $G$-equivariant embedding $\varphi\colon M\to\RR^N$, where $G$ acts on $\RR^N$ linearly. By \cref{algebraicfunctions},   $\RR[M]$ is the image of the homomorphism of $\RR$-algebras $\varphi^*:\RR[x_1, \ldots, x_N]\to C^0(M)$. Thus $\RR[M]$ is a finitely-generated $\RR$-algebra. It is dense in $C^0(M)$ in the supremum norm by  \Cref{supnorm}.
%Let $f\in C^0(M)$ and $\varepsilon>0$. Let $G$ denote the isometry group of $M$, fix $p\in M$, and consider $\pi\colon G\to M$ given by $p\mapsto g\cdot p$. By the Peter--Weyl Theorem (see \cite[III Theorem 3.1]{BtD}), there exists a $G$-finite (that is, representative) function $F\colon G\to\RR$ which is $\varepsilon$-close to $f\circ \pi$. By averaging $F$ over the isotropy group $G_p$, we may assume $F$ is invariant under $G_p$, and thus descends to a $G$-finite, \todo{more details?} that is, algebraic, function $\bar{F}\colon M\to\RR$. The function $\bar{F}$ is $\varepsilon$-close to $f$, and therefore $\RR[M]$ is dense in $C^0(M)$.
\end{proof}

\subsection{Manifold submetries}\label{prelimsub}

In this section we recall the definition and some basic facts about submetries (respectively manifold submetries), and refer the reader to \cite{KL22} (respectively \cite{MR20}) for more details.

Submetries, originally introduced in \cite{Berestovskii87}, are maps $\sigma\colon Y\to X$ between metric spaces which take closed metric balls to closed metric balls of the same radius. The fibers of $\sigma$ form a partition of $Y$ into pairwise equidistant closed subsets, and every such partition comes from a submetry. We will frequently switch between these two points of view when convenient. A function $f\colon Y\to Z$, where $Z$ is a set, is called \emph{basic} when it factors through $\sigma$. A subset $A\subset Y$ is called \emph{saturated} when it is a union of fibers, that is, when $A=\sigma^{-1}(B)$ for some subset $B\subset X$. 

If $Y$ is an Alexandrov space, so is $X$ (see \cite[Proposition 3.1]{KL22}), and, in this case, $\sigma$ has  a well-defined differential $d_p \sigma\colon T_p Y\to T_{\sigma(p)} X$ at every $p\in Y$, where $T_p Y, T_{\sigma(p)} X$ denote the tangent cones, see \cite[Proposition 3.3]{KL22}.

Manifold submetries were originally defined in \cite{CG16}.
\begin{definition}\label{D:mansub}
Let $(M,g)$ be a Riemannian manifold, and $\sigma\colon M\to X$ a submetry onto a metric space $X$. We say $\sigma$ is a \emph{manifold submetry} if, for every $x\in X$, there exists a non-negative integer $k=k(x)$ such that $\sigma^{-1}(x)$ is a possibly disconnected smooth $k$-dimensional embedded submanifold of $M$. We sometimes refer to the fibers of $\sigma$ as ``leaves''.
\end{definition}
We stress that all connected components of a given fiber must have the same dimension, but that different fibers may have different dimensions.

Examples of manifold submetries include: Riemannian submersions, (singular) Riemannian foliations with closed leaves, and actions by (not necessarily connected) closed subgroups of the isometry group\footnote{In the latter two examples, we have given the equidistant decomposition, and the actual submetry map is given by the natural projection onto the leaf/orbit space.}.

\begin{definition}
Let $\sigma\colon M\to X$ be a manifold submetry. The \emph{regular part} in $M$, denoted $M_{\reg}$, is given by the union of all fibers of maximal dimension. 
\end{definition}
The regular part $M_{\reg}$ is an open and saturated subset of $M$, 
%, and its complement is a closed subset with vanishing $(n-1)$-dimensional Hausdorff measure, where $n=\dim(M)$, see \cite[Remark 46, Lemma 47]{MR20}. In particular $M_{\reg}$ is 
which is full-measure, hence dense.

\begin{definition}\label{D:bmc}
Let $\sigma\colon M\to X$ be a manifold submetry. We say $\sigma$ has \emph{basic mean curvature} when, for every $p,q\in M_{\reg}$ with $\sigma(p)=\sigma(q)=x\in X$, we have $d_p\sigma(H_p)=d_q\sigma(H_q)$, where $H$ denotes the mean curvature vector field of the fiber $\sigma^{-1}(x)\subset M$.
\end{definition}

%
%\subsection{Spectral Geometry}
%[Recall basic stuff from \cite{LR18}, \cite{LMR23}, ...: averaging, basic eigenvalues and eigenfunctions] 

\section{Laplacian algebras}
In this section we prove purely algebraic facts about sub-algebras $A\subset \RR[M]$, which are preserved under the Laplacian (which we sometimes call ``Laplacian algebras''), where $M$ is a compact homogeneous space. We will always assume $A$ contains the constants $\RR=E_0$. These will eventually be used in \cref{S:fromlaplacian} to produce a manifold submetry $\sigma$ whose algebra of basic algebraic functions $\B(\sigma)$ is equal to $A$.

\begin{definition}
Let $M$ be a compact homogeneous space, and $Z$ a set of real-valued functions on $M$. Define an equivalence relation on $M$ by $p\sim_Z q$ when $f(p)=f(q)$ for all $f\in Z$. We denote by  $\L(Z)$ the partition of $M$ into the equivalence classes.
\end{definition}
Note that $Z$ can be enlarged to the $\RR$-algebra generated by $Z$, without changing $\L(Z)$. On the other hand, $Z$ can always be shrunk to a finite set by the following, whose proof is taken from \cite[Theorem 2.4.8]{DK}:
\begin{lemma}\label{lemmasep}
Let $M$ be a compact homogeneous space, and $A\subset \RR[M]$ a sub-algebra. Then there exists a finite separating set $\rho_1, \ldots, \rho_k\in A$, in the sense that $\L(A)=\L(\{\rho_1, \ldots, \rho_k\})$. The elements $\rho_1, \ldots, \rho_k\in A$ can be chosen to belong to any given basis of $A$.
\end{lemma}
\begin{proof}
Let $B\subset A$ be an $\RR$-basis of $A$. Consider the ideal $I$ in $\RR[M]\otimes_{\RR}\RR[M]$ generated by the elements $1\otimes f - f\otimes 1$, where $f$ runs through all elements of $B$. It follows from \cref{R[M]} and Hilbert's Basis Theorem that $\RR[M]\otimes_\RR\RR[M]$ is Noetherian, and therefore there exist $\rho_1, \ldots, \rho_k\in B$ such that 
\[1\otimes \rho_1 - \rho_1\otimes 1\ , \ \ldots\ , \ 1\otimes \rho_k - \rho_k\otimes 1\] 
generate $I$. Note that $(1\otimes f - f\otimes 1)(p,q)=f(q)-f(p)$, which implies that $\L(A)=\L(\{\rho_1, \ldots, \rho_k\})$.\end{proof}

\begin{definition} Let $M$ be a compact homogeneous space. Given a partition $\F$ of $M$, we denote by $\B(\F)$ the algebra of all elements of $\RR[M]$ which are constant on the elements of $\F$. Given an $\RR$-sub-algebra $A\subset \RR[M]$, we say $A$ is \emph{maximal} when $A=\B(\L(A))$, that is, when for any sub-algebra $A'\subset\RR[M]$ strictly containing $A$, the partition $\L(A')$ is strictly finer than $\L(A)$.
\end{definition}

The following is a generalization of \cite[Theorem A]{MR23}:
\begin{lemma} \label{maximal} Let $M$ be a compact homogeneous space. Let $A\subset \RR[M]$ be an $\RR$-sub-algebra which is preserved by the Laplace--Beltrami operator $\Delta$. Then $A$ is maximal.
\end{lemma}
\begin{proof}
By the Laplacian condition, we know $A=\bigoplus_\lambda (A\cap E_\lambda)$, where $E_\lambda$ denotes the $\lambda$-eigenspace of the Laplacian. Let $\sim_A$ denote the equivalence relation on $M$ modulo $A$, $\sigma\colon M\to X=M/\!\sim_A$ the canonical projection, and $A'=\mathcal{B}(\sigma)$ the algebra of algebraic functions which are $\sigma$-basic, that is, that descend to $X$. Then we can define the map $\varphi\colon A\to C^0(X)$ by $\varphi(f)(\sigma(p)):=f(p)$, and $\varphi(A)$ is an algebra separating points in $X$ (in particular $X$ is Hausdorff), and containing the constants, thus by the Stone--Weierstrass theorem (see, e.g. \cite[Theorem 7.32 on page 162]{Rudin}), it is dense in $C^0(X)$ with respect to the supremum norm.

Suppose, for the sake of contradiction, that there is some $f'\in A'\setminus A$. Since $f'\in \RR[M]$, it is contained in a finite direct sum of eigenspaces of $\Delta$, that is, $f'\in E:=\sum_{i=1}^NE_{\lambda_i}$. Since $A\cap E\subseteq A$ is finite-dimensional, $p_{A\cap E}(f')$ exists (where the projection is taken with respect to the $L^2$ norm) and $f:=f'-p_{A\cap E}(f')$ is perpendicular to $A$ in the $L^2$ norm, and$f\neq 0$. By the density of $\varphi(A)$, there is $g\in A$ such that $\|f-g\|_{C^0(M)}=\|\varphi(f)-\varphi(g)\|_{C^0(X)}<\varepsilon$ for any choice of $\varepsilon>0$. Using the Cauchy--Schwarz inequality, we obtain
\[
\langle f,g\rangle= \langle f,f\rangle-\langle f,f-g\rangle\geq \|f\|_{L^2}^2-\|f\|_{L^2}\cdot \|f-g\|_{L^2}
\]
Since $\|f-g\|_{L^2}\leq \sqrt{\vol(M)} \ \|f-g\|_{C^0(M)}$, we can choose $g\in A$ such that $\langle f,g\rangle >0$,  which contradicts the fact that $f\perp A$.

Therefore $A'=A$, that is, $A$ is maximal.
\end{proof}

\begin{definition} Let $A\subset B$ be $\RR$-algebras. An $\RR$-linear projection map $\pi\colon B\to A$ is called a \emph{Reynolds operator} when $\pi(ab)=a\pi(b)$ for all $a\in A$ and $b\in B$.

\end{definition}

\begin{lemma}\label{Reynolds}
 Let $M$ be a compact homogeneous space, and $A\subset \RR[M]$ be a sub-algebra preserved by $\Delta$. Then there exists a Reynolds operator $\pi\colon\RR[M]\to A$.
\end{lemma}
\begin{proof}
Since $A$ is preserved by the Laplacian $\Delta$, we have $A=\bigoplus_\lambda (A\cap E_\lambda)$, where $E_\lambda$ denotes the $\lambda$-eigenspace of the Laplacian. For each $\lambda$, let $\pi_\lambda\colon E_\lambda\to A\cap E_\lambda$ be the $L^2$-orthogonal projection, and define $\pi\colon \RR[M]\to A$ by $\pi=\oplus_{\lambda}\pi_{\lambda}$. Since eigenspaces are pairwise $L^2$-orthogonal, it follows that $\pi$ is the $L^2$-orthogonal projection of $\RR[M]$ onto $A$, in the sense that $\ker(\pi)=\oplus_\lambda \ker(\pi_\lambda)$ is orthogonal to $A$, and $\RR[M]=A\oplus \ker(\pi)$.

Let $a\in A$ and $b\in \RR[M]$. Write $b=b_1+b_2$, where $b_1=\pi(b)\in A$ and $b_2 \perp A$. Then 
\[\pi(ab)=\pi(a b_1 +a b_2)=ab_1 +\pi(a b_2)=a\pi(b) +\pi (a b_2),\]
where, in the second equality, we have used that $A$ is an algebra, hence $a b_1 \in A$.
Thus, to show $\pi$ is a Reynolds operator, it suffices to show that $a b_2\perp A$. Let $f\in A$. Then
\[\langle f, a b_2\rangle=\int f a b_2 = \langle f a, b_2\rangle =0\]
because $f a\in A$ (since $A$ is an algebra) and $b_2\perp A$.
\end{proof}

\begin{lemma}\label{FF}Let $M$ be a compact homogeneous space, and $A\subset \RR[M]$ be a sub-algebra preserved by $\Delta$. Then $A=\RR[M]\cap F(A)$, where $F(A)$ denotes the field of fractions of $A$.
\end{lemma}
\begin{proof}
Since a Reynolds operator exists by \cref{Reynolds}, the proof in \cite[Lemma 24(b)]{MR20} applies. Indeed, one inclusion is clear, and for the other, let $f,g\in A$ and $h\in \RR[M]$  such that $h=f/g$. Then $f=gh$, and thus, applying a Reynolds operator $\pi$ on both sides yields $f=g\pi(h)$, which implies $f/g\in A$.
\end{proof}

\section{Basic algebraic functions}
In this section we introduce the algebra of basic algebraic functions $\B(\sigma)$ associated to a manifold submetry $\sigma$ on a compact homogeneous space, which generalizes the algebra of invariant functions from classical Invariant Theory. We show (\cref{thmalgebraic} below) that, when the manifold submetry  $\sigma$ has basic mean curvature, this algebra is ``large enough'' to separate fibers, and hence  that $\sigma$ is algebraic, which eventually yields \cref{normal}(a) from the Introduction.

\begin{definition}
Let $(M,g)$ be a compact homogeneous space, and $\sigma\colon M\to X$ be a manifold submetry (see \cref{prelimsub}). The algebra of \emph{basic algebraic functions} consists of all algebraic functions (see \cref{prelimhom}) which are constant on the fibers of $\sigma$, and is denoted by $\B(\sigma)\subset \RR[M]$.
\end{definition}

\begin{definition}
Let $(M,g)$ be a compact homogeneous space, and $\sigma\colon M\to X$ be a manifold submetry. We say $\sigma$ is \emph{algebraic} when there are basic eigenfunctions $\rho_1, \ldots, \rho_k\in\B(\sigma)$ such that, whenever $p,q\in M$ are such that $\sigma(p)\neq \sigma(q)$, there exists $i$ such that $\rho_i(p)\neq \rho_i(q)$. A set $\{\rho_1, \ldots, \rho_k\}\subset\B(\sigma)$ satisfying this property is called a (finite) \emph{separating set} for $\sigma$.
\end{definition}

The following clarifies the definition and immediate consequences of algebraicity of manifold submetries.
\begin{proposition}\label{P:algebraic}
Let $(M,g)$ be a compact homogeneous space, and $\sigma\colon M\to X$ be a manifold submetry. Assume $\sigma$ is algebraic, with separating set $\rho_1, \ldots, \rho_k\in\B(\sigma)$. Let $\rho=(\rho_1, \ldots, \rho_k)$. Then:
\begin{enumerate}[(a)]
 \item The map $\rho\colon M\to \RR^k$  descends to a homeomorphism $X\to\rho(M)$.
 \item $\rho(M)$ is a semi-algebraic subset of $\RR^k$. 
 \item Every $\sigma$-fiber is an algebraic variety in the sense that it corresponds to an algebraic variety in $\RR^N$ under any $G$-equivariant embedding $M\to\RR^N$ where $G$ acts linearly on $\RR^N$.
\end{enumerate}
\end{proposition}
\begin{proof}
(a) From the definition of submetry, the metric topology on $X$ coincides with the quotient topology associated with $\sigma\colon M\to X$. Thus the induced map $\bar{\rho}\colon X\to\rho(M)$ is continuous. Since $\bar{\rho}$ is bijective, $X$ is compact, and $\rho(M)$ is Hausdorff, we conclude that $\bar{\rho}$ is a homeomorphism.

To prove (b), we use \cref{algebraicfunctions}. Let $G$ be a compact Lie group acting transitively on $M$, and $\varphi\colon M\to\RR^N$ be a $G$-equivariant embedding, where $G$ acts linearly on $\RR^N$. Then $\rho_i=P_i\circ \varphi$ for certain polynomials $P_i\colon \RR^n\to\RR$, and $\rho(M)=P(\varphi(M))$, where $P\colon\RR^N\to\RR^k$ is given by $P=(P_1, \ldots, P_k)$. Since $\varphi(M)$ is a $G$-orbit, it is an algebraic variety in $\RR^N$. Thus $P(\varphi(M))$ is a semi-algebraic subset, by Tarski's Quantifier Elimination Theorem.

(c) is clear from the proof of (b).
\end{proof}

The main result in this section is the following theorem, most of whose proof follows along the same lines as that of \cite[Theorem 1.1]{LR18} and \cite[Theorem 19]{MR20}. The notable exception is the last step, namely existence of a finite separating set, cf. \cref{separatingset}.
\begin{theorem}\label{thmalgebraic}
Let $(M,g)$ be a compact homogeneous space, and $\sigma\colon M\to X$ be a manifold submetry. Assume $\sigma$ has basic mean curvature. Then $\B(\sigma)$ is a sub-algebra of $\RR[M]$ preserved by the Laplace--Beltrami operator $\Delta$, and $\sigma$ is algebraic.
\end{theorem}

We note that compactness of $M$ is necessary in \cref{thmalgebraic}:
\begin{example}\label{E:non-algebraic}
Consider $M=\RR$ and $\sigma$ given by the quotient map $\RR\to\ \RR/\Z$. The fibers (orbits) are infinite discrete subsets, hence not algebraic varieties. 

For a similar example with connected fibers, consider $M=\RR^3$ and $\sigma$ given by the quotient map $\RR^3\to\ \RR^3/\RR$, where $\RR$ acts on $\RR^3$ by isometries via a ``corkscrew'' motion, where $t\in\RR$ is mapped to a rotation by $t$ in the $xy$-plane (fixing the $z$-axis) followed by translation by $t$ in the direction of the $z$-axis. Then, the orbit through any $p\in\RR^3$ not on the $z$-axis is a helix which is not an algebraic variety because its intersection with any plane containing the $z$-axis is an infinite discrete subset. 
\end{example}

Before providing a proof of \cref{thmalgebraic}, we recall the definition and basic facts regarding the averaging operator.

\begin{definition}
Let $M$ be a Riemannian manifold, and $\sigma\colon M\to X$  a manifold submetry with compact fibers. The \emph{averaging operator} is the linear map $\Av\colon L^2(M)\to L^2(M)$ defined by
\[\Av(f)(p)=\frac{1}{\vol(L_p)}\int_{q\in L_p}f(q) \ d\vol_{L_p}(q).\]
Here $L_p$ denotes the $\sigma$-fiber through $p$;  $d\vol_{L_p}$ denotes the Riemannian volume form on $L_p$ induced by the restricted Riemannian metric from $M$; and $\vol(L_p)=\int_{q\in L_p}1 \ d\vol_{L_p}(q)$ denotes the corresponding volume.
\end{definition}

\begin{lemma}\label{average}
Let $\sigma\colon M\to X$ be a manifold submetry with compact fibers and basic mean curvature. Then the following hold:
\begin{enumerate}[(a)]
\item The averaging operator is well-defined, and is the $L^2$-orthogonal projection onto the subspace of $L^2(M)$ consisting of the elements with a basic representative. 
\item For every $f\in C^{\infty}(M)$, the element $\Av (f)\in L^2(M)$ has a smooth representative.
\item The operators $\Delta, \Av \colon C^{\infty}(M)\to C^{\infty}(M)$ commute.
\end{enumerate}
\end{lemma}
\begin{proof}
The proofs of these statements are analogous to the proofs in the special case where $\sigma$ is given by a singular Riemannian foliation, see \cite[Section 3]{LR18} (see also \cite{PR96}), so we only sketch the arguments and point out the key fact about manifold submetries that is needed. Namely: the complement of the regular part $M_{\reg}$ (that is, the union of the fibers of maximal dimension) is a closed subset which is the union of 
%locally finite set of
 submanifolds of codimension at least two. This follows from the results on stratification in \cite[Appendix B]{MR20}. More precisely, by \cite[Remark 46]{MR20}, $M_{\reg}$ is open, and by \cite[Proposition 43]{MR20} every stratum is a smooth submanifold. By \cite[Lemma 47]{MR20}, there are no strata of codimension one, thus $M\setminus M_{\reg}$ is the union of strata of codimension at least two.% Local finiteness around a point $p\in M$ follows from the Homothetic Transformation Lemma (\cite[Lemma 42]{MR20}) and induction, since there is a one-to-one correspondence between strata meeting a given distinguished tubular neighborhood around a fiber $L_p$ (except possibly for $L$ itself), and strata in the infinitesimal submetry $\sigma_p$ defined on the unit sphere in the normal space $\nu_p L_p$ (see \cite[Section 3.2]{MR20}).

Since $M_{\reg}$ has full-measure, the results analogous to \cite[Sections 3.1 and 3.2]{LR18} are proved in essentially the same way. In particular this yields (a), and the fact that the averaging operator and the Laplace--Beltrami operator commute for smooth functions $M_{\reg}\to\RR$.

The proof of (b) follows the same arguments as \cite[Theorem 3.3]{LR18}, where the key fact about $M\setminus M_{\reg}$ used is that it has codimension at least two in $M$, in particular it has vanishing $(n-1)$-dimensional Hausdorff measure.

(c)  follows from (b) and the fact that $\Av$ and $\Delta$ commute for functions on $M_{\reg}$.
\end{proof}

\begin{proof}[Proof of \cref{thmalgebraic}]
The subset $\B(\sigma)\subset\RR[M]$ is an $\RR$-sub-algebra by definition.

Let $f\in \B(\sigma)$. Then $f$ is a finite sum of eigenfunctions by definition, and hence $\Delta(f)$ is also a finite sum of eigenfunctions. Moreover, $\Delta(f)$ is basic by the definition of the averaging operator and by \cref{average}. Indeed, a function is basic if and only it is equal to its average, and $\Av(\Delta (f))=\Delta (\Av(f))=\Delta(f)$. Therefore $\B(\sigma)$ is preserved by $\Delta$.

Next we show that $\B(\sigma)$ separates $\sigma$-fibers. Let $L, L'$ be two distinct fibers, and choose $f\in C^0(M)$ which is constant equal to $0$ on $L$ and constant equal to $1$ on $L'$. After averaging, we may assume $f$ is basic. Choose $\varepsilon \in(0,1/2)$. By \cref{R[M]}, there is $h\in\RR[M]$ which is $\varepsilon$-close to $f$ in the supremum norm. Thus $\Av(h)$ is $\varepsilon$-close to $f$, because $f$ is basic, and hence $\Av(h)$ has different values on $L$ and $L'$. Moreover, $\Av(h)\in\B(\sigma)$ because $\Delta,\Av$ commute. Therefore  $\B(\sigma)$ separates $\sigma$-fibers.

Since $\B(\sigma)$ is preserved by $\Delta$, it has an $\RR$-basis of eigenfunctions. By \cref{lemmasep},  $\sigma$ is algebraic.
\end{proof}

\begin{remark}\label{separatingset}
If $\B(\sigma)$ were finitely generated, the last part of the proof above would be simplified, because any generating set $\rho_1, \ldots, \rho_k$ would be a separating set. Since we do not know if $\B(\sigma)$ is finitely generated (see \cref{Q:fg}), we have instead employed \cref{lemmasep}.
\end{remark}

\begin{remark}
\cref{thmalgebraic} can be partially generalized to the situation where $(M,g)$ is assumed to be a \emph{real-analytic} (instead of homogeneous) closed Riemannian manifold. Namely, there exists a \emph{finite} set $\rho_1, \rho_2, \ldots, \rho_k$ of eigenfunctions of the Laplace--Beltrami operator which separates $\sigma$-fibers. The proof follows along the same lines as the proof of \cref{thmalgebraic} above, with a couple of modifications. First, the basic eigenfunctions of $(M,g)$ are real-analytic (being solutions of PDE's with real-analytic coefficients, see \cite[\S 2.2]{LMP23}). Second, in the proof of \cref{lemmasep}, one uses the fact (see \cite[Th\'eor\`eme (I,\ 9)]{Frisch67})
%\footnote{More precisely, we take $X$ to be the compact real-analytic manifold $M$, which, is, in particular, a ``real-analytic space''; and $A=M$, which is a semi-analytic subspace of itself. We also note that the ``Stein'' condition is automatically satisfied in the real case, see the first paragraph of \cite[page 123]{Frisch67}.}
that the ring of all globally-defined real-analytic functions on $M\times M$ is Noetherian.
\end{remark}

\section{From Laplacian algebras to manifold submetries}
\label{S:fromlaplacian}
%Follow steps as in \cite{MR20}, with adaptations/improvements/shortcuts:
%\begin{itemize}
%\item define Riemannian submersion $\rho^{reg}$ almost everywhere on $M$, using a finite set $\rho_1, \ldots \rho_k$ of separating elements of $A$ (cf. \cite[Proposition 26]{MR20})
%\item  $\rho^{reg}$ has basic mean curvature  (cf. \cite[Lemma 31]{MR20})
%\item Equidistance (in $M$) of fibers of $\rho^{reg}$ (cf. \cite[Proposition 27]{MR20})
%\item Completion $\hat{\rho}$ is a submetry defined everywhere on $M$ (cf. \cite[Proposition 28]{MR20})
%\item $\hat{\rho}$ is a manifold submetry with basic mean curvature (cf.  \cite[Propositions 29, 30, 32]{MR20}). Possible shortcut avoiding \cite[Propositions 29, 30]{MR20}: use \cite[Theorem 1.4 or Proposition 1.8]{LW24}
%\item $A$ and $\hat{A}:=\B(\hat{\rho})$ have same field of fractions, hence are equal (cf. \cite[Proof of Theorem 25]{MR20})
%\end{itemize}

The main goal of this section is to identify, given a compact homogeneous space $M$, the sub-algebras of the algebra $\RR[M]$ of algebraic functions which arise as the algebra of basic algebraic functions $\B(\sigma)$ for some  manifold submetry $\sigma$ with basic mean curvature. They are precisely the sub-algebras preserved by the Laplace--Beltrami operator. This eventually leads to the one-to-one correspondence in \cref{normal}(b).
\begin{theorem} \label{T:fromlaplacian}
Let $M$ be a compact homogeneous space, and $A\subset \RR[M]$ a sub-algebra preserved by the Laplacian. Then there exists a manifold submetry $\sigma\colon M\to X$ with basic mean curvature, such  that $\B(\sigma)=A$. 
\end{theorem}

Before presenting the proof of \cref{T:fromlaplacian}, we use it to prove the following:
\begin{corollary}\label{1-1correspondence}
Let $M$ be a compact homogeneous space. Define an equivalence relation $\sim$ on the set of all manifold submetries from $M$ with basic mean curvature by $\sigma\sim\sigma'$ if and only if the fibers of  $\sigma$ and $\sigma'$ define the same decomposition of $M$. Then the map $\sigma\mapsto\B(\sigma)$ defines a bijection from the set of all manifold submetries from $M$ with basic mean curvature (modulo $\sim$), to the set of all sub-algebras of $\RR[M]$ that are preserved under the Laplace--Beltrami operator.
\end{corollary}
\begin{proof}
Let $\sigma\colon M\to X$ be a manifold submetry with basic mean curvature. By \cref{thmalgebraic}, $\B(\sigma)$ is a sub-algebra of $\RR[M]$ preserved by the Laplacian. By definition,  $\B(\sigma)$ depends only on the decomposition of $M$ into the fibers of $\sigma$. Thus the map above is well-defined.

By \cref{T:fromlaplacian}, the image of $\B$ is the set of all sub-algebras of $\RR[M]$ that are preserved under the Laplace--Beltrami operator.

Let $\sigma\colon M\to X$ and $\sigma'\colon M\to X'$ be two manifold submetries with basic mean curvature, and assume $\B(\sigma)=\B(\sigma')$. By \cref{thmalgebraic}, $\sigma$ and $\sigma'$ are algebraic, and thus $\B(\sigma)=\B(\sigma')$ separates $\sigma$-fibers and separates $\sigma'$-fibers. Therefore, given $p,q\in M$, they belong to different $\sigma$-fibers if and only if they belong to different $\sigma'$-fibers. We conclude that $\sigma\sim\sigma'$.
\end{proof}

%\begin{proof}[Proof of \cref{T:fromlaplacian}]
The remainder of this section is devoted to a proof of \cref{T:fromlaplacian}, and includes several lemmas. Some of the statements are analogous to results in \cite{MR20} (which treats the spherical case), and we skip details for those where the \emph{proofs} are also analogous. The main novelties are the inclusion of \cref{L:stratification}, since we do not know if  $\rho(M_{\reg})$ is a smooth manifold when $M$ is not a sphere (cf. \cite[Proposition 26]{MR20}\footnote{In the preparation of the present article we have noticed that one step in the proof of  \cite[Proposition 26]{MR20} is not sufficiently justified. The missing argument can be shown to achieve the result as stated in \cite[Proposition 26]{MR20}, but this argument only applies to the case where $M$ is a round sphere, because it uses the finite generation of the Laplacian algebra, cf. \cref{Q:fg}.}), and the use of \cite{LW24} in the paragraph following \cref{L:completion}, which replaces arguments similar to \cite[Propositions 29 and 30]{MR20}.

%\begin{proof}[Proof of  \cref{T:fromlaplacian}]
\subsection{Proof of  \cref{T:fromlaplacian}}
Fix $(M^n,g)$ a compact homogeneous space, $A\subset\RR[M]$ a sub-algebra preserved by $\Delta$, and $\rho_1, \ldots, \rho_k$ a finite separating set (which exists by \cref{lemmasep}).

Define $\rho=(\rho_1, \ldots, \rho_k)\colon M\to \RR^k$. Let  $m$ be the maximal dimension of the span of $\nabla\rho_1, \ldots, \nabla \rho_k$, and define 
\[M_{\reg}=\{p\in M \ :\  \dim(\operatorname{span}(\nabla\rho_1(p), \ldots, \nabla \rho_k(p)))=m \}.\]

Note that, for $X\in T_pM$, we have $d\rho_p(X)=(\langle \nabla\rho_1(p),X\rangle, \ldots , \langle \nabla\rho_k(p),X\rangle)$, so that $\ker(d\rho_p)=\{ \nabla\rho_1(p),\ldots,  \nabla\rho_k(p)\}^\perp$. This implies that $\rho|_{M_{\reg}}$ is a constant rank map, and the normal space to the fiber through $p\in M_{\reg}$ is  
\[\nu_p(\rho^{-1}(x))=\operatorname{span}(\nabla\rho_1(p), \ldots, \nabla \rho_k(p)).\]

\begin{lemma}\label{L:stratification}
There exists a polynomial function $Q\colon \RR^k\to\RR$ such that $X_0=\rho(M_{\reg})\setminus Z(Q)$ is a smooth $m$-dimensional embedded submanifold of $\RR^k$, and $M_0=\rho^{-1}(X_0)$ is an open, full-measure subset of $M_{\reg}$, where $Z(Q)$ denotes the zero set of $Q$. 
\end{lemma}
\begin{proof}
We first show that the set $\rho(M_{\reg})\subset \RR^k$ is semi-algebraic. The proof is analogous to the proof of \cref{P:algebraic}(b). We use a $G$-equivariant embedding $\varphi\colon M\to \RR^N$, where $G$ is a compact Lie group acting transitively on $M$ and linearly on $\RR^N$, see \cref{algebraicfunctions}. In particular, $\varphi(M)$ is a $G$-orbit, hence an algebraic variety in $\RR^N$. Choose polynomial functions $P_i\colon \RR^N\to\RR$ such that $\rho_i=P_i\circ \varphi$. Then $\varphi(M_{\reg})$ is the subset of $\varphi(M)$ where the $1$-forms $dP_i$, restricted to the tangent space of $\varphi(M)$, span an $m$-dimensional space. Since the $P_i$ are polynomials, and the tangent space of $\varphi(M)$ is spanned by a finite set of linear vector fields, namely the action fields corresponding to a basis of the Lie algebra of $G$, it follows that $\varphi(M_{\reg})$ is semi-algebraic. Therefore, putting $P=(P_1, \ldots, P_k)$, we conclude that  $\rho(M_{\reg})=P(\varphi(M_{\reg}))\subset \RR^k$ is semi-algebraic by Tarski's Theorem.

Recall that every semi-algebraic subset $E\subset\RR^k$ admits a semi-algebraic stratification, see \cite[Definition 2.4.1 on page 68 and Proposition 2.5.1 on page 75]{BenedettiRisler}. This means that $E$ is the finite disjoint union of ``strata'',  where each ``stratum'' is a semi-algebraic locally closed analytic submanifold of $\RR^k$ (and they satisfy the ``frontier condition''). The dimension of $E$ is defined as the maximum dimension (as a smooth manifold) of its strata, see \cite[Definition 2.5.3 on page 75]{BenedettiRisler}.

Since  $\rho|_{M_{\reg}}$ is a smooth map with constant rank equal to $m$, the semi-algebraic set $\rho(M_{\reg})\subset \RR^k$ has dimension $m$. %Let $\rho(M_{\reg})=\sqcup_{i=1}^r A_i$ be a stratification. 
Take a stratification of $\rho(M_{\reg})$ and let $T\subset \rho(M_{\reg})$ be the union of the strata with dimension strictly smaller than $m$. In particular, $\rho(M_{\reg})\setminus T$ is a smooth $m$-dimensional embedded submanifold of $\RR^k$.

Since $T$ is a semi-algebraic set with dimension smaller than $m$, \cite[Lemma 3.4.4 on page 136]{BenedettiRisler} implies that $T\subset Z$, where $Z$ is an \emph{algebraic} set of dimension (as a semi-algebraic set, i.e., in the sense of \cite[Definition 2.5.3]{BenedettiRisler}) smaller than $m$. Thus $Z$ is a finite union of locally closed smooth submanifolds of dimension smaller than $m$. Letting $Q$ be a polynomial defining $Z$ finishes the proof.
%By \cite[Proposition 2.5.1 on page 75]{BenedettiRisler}, it admits a finite semi-algebraic stratification. By \cite[Definition 2.4.1 on page 68]{BenedettiRisler}, this implies $\rho(M)=\sqcup_{i=1}^r A_i$, where each ``stratum'' $A_i$ is a semi-algebraic locally closed smooth submanifold of $\RR^k$. %, and they satisfy the ``frontier condition''.
%By \cite[Lemma 3.4.4 on page 136]{BenedettiRisler} the dimension of each $A_i$ as a smooth submanifold coincides with the Zariski dimension of its Zariski closure $V_i$. Moreover, the maximum dimension of the $A_i$ is $m$. By re-ordering if necessary, we  may assume that $\dim(A_i)<m$ for $i=1, \ldots, s$, and $\dim(A_i)=m$ for $i>s$. Let $Q_i$ be a defining polynomial for $V_i$, and define $Q=\prod_{i=1}^s Q_i$. Then $Z(Q)=\cup_{i=1}^s V_i$ contains $\sqcup_{i=1}^s A_i$, and thus $\rho(M)\setminus Z(Q)$ is a smooth embedded $m$-dimensional submanifold. Since $Z(q)$ is a union of subvarieties of dimension less than $m$, $\rho(M)\setminus Z(Q)$ is open and has full-measure in $\rho(M)$. This implies $\rho^{-1}(X_0)$ is an open, full-measure subset of $M$. \todo{justify better}
\end{proof}

 Since $\rho|_{M_0}\colon M_0\to X_0$ is a constant rank map, it is a surjective smooth submersion.

\begin{lemma} (cf. \cite[Proposition 26]{MR20})
There exists a smooth metric $b$ on $X_0$ such that $\rho|_{M_0}\colon M_0\to X_0$ is a Riemannian submersion with basic mean curvature.
\end{lemma}
\begin{proof}
For all $i,j$, we have 
\[\langle \nabla \rho_i, \nabla\rho_j\rangle= \frac{1}{2}\left( \Delta(\rho_i \rho_j) - (\Delta\rho_i)\rho_j - \rho_i(\Delta\rho_j)\right)\in A,\]
because $A$ is a sub-algebra  preserved by the Laplacian. Thus, for all $p,q\in M_{\reg}$ with $\rho(p)=\rho(q)$, we have $\langle \nabla \rho_i (p), \nabla\rho_j(p)\rangle=\langle \nabla \rho_i(q), \nabla\rho_j(q)\rangle$ by the separation property.

Thus, for every $x\in X_0$, there is an inner product $b_x$ on $T_x X_0$ with respect to which $d\rho_p$ restricted to the normal space  $\nu_p(\rho^{-1}(x))$ is an isometry onto $T_x X_0$, for every choice of $p\in \rho^{-1}(x)$. This defines a metric tensor $b$ on $X_0$ such that $\rho|_{M_0}\colon M_0\to X_0$ is a Riemannian submersion. Since $\rho|_{M_0}\colon M_0\to X_0$ is a smooth submersion, and the metric $g$ is smooth, it follows that $b$ is also smooth.

The proof that the mean curvature vector field is basic is the same as in \cite[Lemma 31]{MR20}. 
\end{proof}

\begin{lemma} (cf. \cite[Proposition 27]{MR20})
Every pair $L_1, L_2$ of $\rho$-fibers of the Riemannian submersion $\rho|_{M_0}\colon M_0\to X_0$ are equidistant in $M$.
\end{lemma}
\begin{proof}
Let $p,q\in L_1$. Let $\gamma$ be a (unit speed) minimizing geodesic in $M$ between $p=\gamma(0)$ and $L_2$, of length $\ell$. Then $\gamma'(t)$ is horizontal at $t=\ell$, hence for $t$ near $\ell$. 

Write the vectors $\gamma'(t)$ and $\nabla \rho_i (\gamma(t))$, for $i=1, \ldots k$ in a parallel frame along $\gamma$. Arrange the vectors thus obtained as the columns of an $n\times(k+1)$ matrix $B(t)$, and let $F(t)$ be the sum of the squares  of all the $(m+1)\times (m+1)$ minors of $B(t)$. Thus, for all $t$ such that $\gamma(t)\in M_0$, we have that $F(t)=0$ if and only if $\gamma'(t)$ is horizontal.

Since $(M,g)$ is a compact homogeneous space, it is a real analytic manifold, and thus its geodesics are real analytic, and so are the eigenfunctions of the Laplacian. Therefore $F(t)$ is analytic. Since $F(t)$ vanishes for $t$ near $\ell$, it must vanish everywhere, in particular near $t=0$. That is, the restriction of $\gamma(t)$ to a small interval $(-\varepsilon, \varepsilon)$ is a horizontal geodesic for the Riemannian submersion $\rho|_{M_0}\colon M_0\to X_0$.

Let $\hat{\gamma}\colon \RR\to M $ be the geodesic such that $\hat{\gamma}|_{(-\varepsilon, \varepsilon)}$ is the horizontal lift of $\rho\circ\gamma|_{(-\varepsilon, \varepsilon)}$ with $\hat{\gamma}(0)=q$. Then $\rho(\gamma(t))=\rho(\hat{\gamma}(t))$ for all $t\in (-\varepsilon, \varepsilon)$, and therefore for all $t\in\RR$, because these are analytic functions. In particular, $\hat{\gamma}(\ell)\in L_2$, which implies that $d_M(q,L_2)\leq \ell=d_M(p, L_2)$.

Reversing the roles of $p,q$, we obtain $d_M(q,L_2)=d_M(p, L_2)$. Reversing the roles of $L_1, L_2$, we conclude that $L_1, L_2$ are equidistant.
\end{proof}

\begin{lemma}\label{L:completion}
There exists a submetry $\hat{\rho}\colon M\to \hat{X}$ such that  $M_0$ is a saturated set, and, on $M_0$, the decomposition into the fibers of $\hat{\rho}$ coincides with the decomposition into the fibers of the Riemannian submersion $\rho|_{M_0}$.
\end{lemma}
\begin{proof}
Using the previous lemmas, this is analogous to the proof of \cite[Prop. 28]{MR20}.
\end{proof}

Since almost all regular fibers of the submetry $\hat{\rho}\colon M\to \hat{X}$ have basic mean curvature, we can use \cite[Proposition 15.2]{LW24} to conclude that every fiber of $\hat{\rho}$ is a smooth submanifold. According to the definitions in \cite{LW24}, this still leaves the possibility that some fibers have connected components of different dimensions. But  basic mean curvature implies, via an argument analogous to \cite[Proposition 32]{MR20}, that every fiber of $\hat{\rho}$ does indeed have components of the same dimension, that is, $\hat{\rho}$ is a manifold submetry in the sense of \cref{D:mansub}.

Note that the regular part of $\hat{\rho}$ in the sense of \cref{prelimsub} contains $M_0$, but may be larger. Nevertheless, since $M_0$ is dense in $M$, it follows that $\hat{\rho}$ has basic mean curvature in the sense of \cref{D:bmc}. Thus, by \cref{thmalgebraic}, the sub-algebra $\hat{A}=\B(\hat{\rho})\subset\RR[M]$ is preserved by the Laplacian, and separates $\hat{\rho}$-fibers, so that $\L(\hat{A})$ coincides with the decomposition of $M$ into the fibers of $\hat{\rho}$.

It remains to show that $A=\hat{A}$.

On the one hand, let $f\in A$. The construction of $\hat{\rho}$ implies that, for any $p,q\in M$ with $\hat{\rho}(p)=\hat{\rho}(q)$, there exist sequences $p_i, q_i\in M_0$ such that $p=\lim p_i$ and $q=\lim q_i$, and such that  $\rho(p_i)=\rho(q_i)$ for every $i$. Thus  $f(p_i)=f(q_i)$ for all $i$ by the separation property of $\rho_1, \ldots, \rho_k$, and therefore, $f(p)=f(q)$. Thus $f$ is basic with respect to $\hat{\rho}$, that is, $f\in\hat{A}$, and this shows that $A\subset \hat{A}$.

For the reverse inclusion, let $f\in \hat{A}$. Consider $P\in A$ given by the sum of the squares of all $m\times m$ minors of the Gram matrix $\langle\nabla\rho_i, \nabla\rho_j\rangle$. Note that $P\neq 0$ and that $P=0$ on $M\setminus M_{\reg}$. Then, $P \cdot(Q\circ\rho) f=0$ on $M\setminus M_0$ (recall that the polynomial $Q$ was produced by \cref{L:stratification}), and, since $\hat{\rho}$ and $\rho$ have the same fibers on $M\setminus M_0$, and $P\cdot (Q\circ\rho)\in A$, we have that $P\cdot (Q\circ\rho)\cdot f$ is constant on the common level sets of $A$. But \cref{maximal} states that $A$ is maximal, and therefore $P\cdot (Q\circ\rho)\cdot f\in A$. In particular, this shows that $f\in F(A)$, where $F(A)$ denotes the field of fractions of $A$. Using \cref{FF}, we obtain that $f\in A$. Therefore $\hat{A}\subset A$.

This finishes the proof of \cref{T:fromlaplacian}.
%\end{proof}

%\begin{proof}[Proof of \cref{T:fromlaplacian}]
%Follow steps as in \cite{MR20}, with adaptations/improvements/shortcuts:
%\begin{itemize}
%\item define Riemannian submersion $\rho^{reg}$ almost everywhere on $M$, using a finite set $\rho_1, \ldots \rho_k$ of separating elements of $A$ (cf. \cite[Proposition 26]{MR20})
%\item  $\rho^{reg}$ has basic mean curvature  (cf. \cite[Lemma 31]{MR20})
%\item Equidistance (in $M$) of fibers of $\rho^{reg}$ (cf. \cite[Proposition 27]{MR20})
%\item Completion $\hat{\rho}$ is a submetry defined everywhere on $M$ (cf. \cite[Proposition 28]{MR20})
%\item $\hat{\rho}$ is a manifold submetry with basic mean curvature (cf.  \cite[Propositions 29, 30, 32]{MR20}). Possible shortcut avoiding \cite[Propositions 29, 30]{MR20}: use \cite[Theorem 1.4 or Proposition 1.8]{LW24}
%\item $A$ and $\hat{A}:=\B(\hat{\rho})$ have same field of fractions, hence are equal (cf. \cite[Proof of Theorem 25]{MR20})
%\end{itemize}
%\end{proof}

\section{Normal homogeneous spaces and the proofs of \cref{normal} and \cref{basicmean}}

The following is key to the proof of \cref{basicmean}, and its use is the main novelty in the present paper. It follows in a relatively straight-forward way from the theory laid out in \cite{HLO06} (initiated earlier by Terng and Thorbergsson, among others, see \cite{TT95}), and we believe it may be of independent interest.
\begin{theorem}
\label{shape}
Let $M$ be a compact normal homogeneous space, $L\subset M$ a smooth submanifold, $p\in L$ a point on $L$, and $v\in\nu_p L$ a unit normal vector. Denote by
\[ \cdots \leq l_2^- \leq l_1^- < 0 < l_1^+ \leq l_2^+ \leq \cdots\]
the focal lengths (repeated with appropriate multiplicities) of $L$ in the direction of $v$. Then the trace of the  shape operator $A_v$ of $L$ in the direction of $v$ is
\[\operatorname{tr}(A_v)=\sum_{k=1}^\infty \left(\frac{1}{l_k^-} + \frac{1}{l_k^+}\right). \]
\end{theorem}
\begin{proof}
By \cite[page 18]{HLO06}, there is a Riemannian submersion $\pi:V\to M$ with minimal fibers, where $V$ is a separable Hilbert space.

Let $\hat{L}=\pi^{-1}(L)$, choose $\hat{p}\in\hat{L}$, and let $\hat{v}\in T_{\hat{p}} \hat{L}$ be the horizontal lift of $v$. 
By \cite[Lemma 6.1]{HLO06},  the focal lengths of $\hat{L}$ in the direction of $\hat{v}$ are the same as for $L$, that is:
\[ \cdots \leq l_2^- \leq l_1^- < 0 < l_1^+ \leq l_2^+ \leq \cdots\]

Since $V$ is a Hilbert space, the shape operator $A_{\hat{v}}$ of $\hat{L}$ in the direction of $\hat{v}$ has eigenvalues
\[\frac{1}{ l_1^-} \leq \frac{1}{l_2^-}\leq \cdots < 0 < \cdots \leq \frac{1}{l_2^+} \leq \frac{1}{l_1^+}\]
see \cite[page 26]{HLO06}.

By \cite[Lemma 5.2]{HLO06}, $\hat{L}$ is regularizable. By definitions in \cite[page 13]{HLO06}, this implies that the series
\[\sum_{k=1}^\infty \left(\frac{1}{l_k^-} + \frac{1}{l_k^+}\right)\]
converges to the ``regularized trace'' of $A_{\hat{v}}$, denoted $\operatorname{tr}_r(A_{\hat{v}})$.

By \cite[Lemma 5.2]{HLO06} (or rather, its proof), we have $\operatorname{tr}(A_v)=\operatorname{tr}_r(A_{\hat{v}})$. Therefore
\[\operatorname{tr}(A_v)=\sum_{k=1}^\infty \left(\frac{1}{l_k^-} + \frac{1}{l_k^+}\right). \]
\end{proof}

\begin{example}\label{E:Euler}
Let $M=S^2$, and $L$ a latitude circle with distance $\varphi\in(0,\pi)$ from the north pole, and let $v$ be a unit normal vector pointing towards the north pole. Then the shape operator $A_v$, and its trace, are equal to $\cot(\varphi)$. The positive focal distances are 
\[\varphi,\ \varphi+\pi,\ \varphi+2\pi, \ldots\]
and the negative focal distances are
\[\ldots,\ \varphi-2\pi,\ \varphi-\pi \]
Thus \cref{shape} states that
\[\cot(\varphi)=\sum_{k=0}^\infty \left( \frac{1}{\varphi+k\pi} + \frac{1}{\varphi-k\pi-\pi} \right)\]
which can be rewritten as
\[\cot(\varphi)=\lim_{N\to\infty}\sum_{n=-N}^N\frac{1}{\varphi+n\pi}\]
This is a well-known identity due to Euler, see \cite[chapter 20]{AZproofs}. See also \cite{wiki}, section called ``Partial fraction expansion''.
\end{example}

\begin{lemma}\label{basicfocaldata}
Let $M$ be a complete Riemannian manifold, and $\sigma\colon M \to X$ a manifold submetry. Then the focal lengths of  fibers of maximal dimension are ``basic''. That is, given a fiber $L=\sigma^{-1}(x)$ of maximal dimension, the following is true. Let $p_i\in L$ and $v_i\in \nu_{p_i}L$, for $i=1,2$, with $d\sigma(v_1)=d\sigma(v_2)$. Then the focal lengths (with multiplicities) of $L$ in the direction of $v_1$ are the same as those in the direction of $v_2$.
\end{lemma}
\begin{proof}
This is essentially the same proof as \cite[Proposition 17]{MR20} and \cite[Proposition 3.1]{AR15}.
\end{proof}

\begin{proof}[Proof of  \cref{basicmean}]
Follows immediately from \cref{shape} and \cref{basicfocaldata}.
\end{proof}

\begin{proof}[Proof of  \cref{normal}]
Follows immediately from \cref{basicmean}, \cref{thmalgebraic}, and \cref{1-1correspondence}.
\end{proof}

%
%Another direction is to drop the real-analyticity condition, and instead try to use properties of nodal sets, i.e., zero sets of eigenfunctions. Indeed, the equivalence relation given by the submetry $\sigma$, and thought of as a subset of $M\times M$, is given by the intersection of an infinite number of nodal sets. What may give some hope is that, locally, eigenfunctions ``behave like'' polynomials. For example, they can only vanish at a point up to finite order, controlled by the eigenvalue, see  Theorem 4.3.7 in M. Levitin, D. Mangoubi, and I. Polterovich, Topics in Spectral Geometry, preliminary version
%dated May 29, 2023. \url{https://michaellevitin.net/Book/}. {\bf End of Update.}

%\begin{question}
%Let $M=\RR^n$, and $\sigma$ a manifold submetry. What structure results similar to \cref{thmalgebraic} can one obtain?
%\end{question}
%For regular foliations one has the classification or structure results by Gromoll--Walschap, Speranca, Weil ... For manifold submetries, there are partial results in the PhD thesis of Boltner (Heintze's student). In particular there exists a fiber which is an affine subspace $\RR^k\subset \RR^n$ (proof is by a ``Soul Theorem''). Is the foliation then obtained by (what kinds of) "screw motions" along $\RR^k$ of (what kinds of) manifold submetries on the normal space(s) of $\RR^k$? 

\bibliography{ref}
\bibliographystyle{alpha}

\end{document}